\documentclass[11pt]{amsart}
\usepackage[margin=1.08in]{geometry}
\usepackage{amsmath,amssymb,amsthm,amscd,mathtools}
\usepackage{enumitem}
\usepackage{xcolor}
\usepackage{tikz}
\usetikzlibrary{arrows.meta,calc,decorations.pathmorphing,positioning}
\usepackage{microtype}
\usepackage[pdfencoding=auto,psdextra,bookmarksnumbered=true,breaklinks=true]{hyperref}
\definecolor{braidred}{RGB}{230,45,55}
\definecolor{braidblue}{RGB}{20,105,220}
\definecolor{braidgreen}{RGB}{30,170,85}
\definecolor{braidturquoise}{RGB}{0,175,190}
\definecolor{braidpurple}{RGB}{145,70,200}
\hypersetup{
  colorlinks=true,
  linkcolor=braidblue!75!black,
  citecolor=braidpurple!80!black,
  urlcolor=braidturquoise!70!black,
  pdfauthor={Anar Akhmedov},
  pdftitle={Braided Multisections and Symplectic Four-Manifolds with the Rational Cohomology of S2 x S2},
  pdfsubject={Symplectic four-manifolds, braided multisections, and mixed fiber sums}
}
\newtheorem{theorem}{Theorem}[section]
\newtheorem{proposition}[theorem]{Proposition}
\newtheorem{lemma}[theorem]{Lemma}
\newtheorem{corollary}[theorem]{Corollary}
\newtheorem{question}[theorem]{Question}
\theoremstyle{definition}

\newtheorem{remark}[theorem]{Remark}

\newcommand{\Int}{\operatorname{int}}

\title[Braided multisections and rational cohomology $S^2\times S^2$'s]
{Braided Multisections and Symplectic Four-Manifolds\\
with the Rational Cohomology of $S^2\times S^2$}

\author{Anar Akhmedov}\date{}

\begin{document}

\begin{abstract}
We construct symplectic four-manifolds by taking mixed fiber sums along
explicit cyclic multisections in ruled surfaces. For a connected unbranched
degree-$p$ multisection in $\Sigma_g\times S^2$, we determine the first
homology and fundamental group of the complement and prove that its boundary
is incompressible. It follows that no direct gluing of two such complements
can be simply connected; moreover, the first homology of every direct sum
retains finite quotients determined by the covering degrees.

We classify the mixed sums having Euler characteristic $4$ and signature
$0$. Up to interchanging the two summands, exactly three possibilities occur,
corresponding to the degree pairs $(2,3)$, $(2,4)$, and $(3,3)$. For each of
these cases, suitable adapted product-framed symplectic gluings have the
rational cohomology ring of $S^2\times S^2$. Varying the gluing by symplectic
transvections produces infinitely many pairwise nondiffeomorphic examples,
distinguished by the unbounded orders of their finite first homology groups.

We also construct the twisted ruled analogue of the $(2,4)$ case.
Explicit finite-holonomy multisections give connected square-zero symplectic
surfaces in the classes $2S_3-F_3$ and $4S_2-2F_2$ in the nontrivial
$S^2$-bundles over $\Sigma_3$ and $\Sigma_2$, respectively.  More generally,
for a square-zero degree-$p$ multisection in the nontrivial bundle the
complement has first homology
$\mathbb Z^{2g}\oplus\mathbb Z/(p/2)$.  Suitable gluings in the twisted
$(2,4)$ case have $b_1=0$, $b_2=2$, and signature zero, and every such
sum is non-spin.  Hence they have the rational cohomology ring of
$\mathbb CP^2\#\overline{\mathbb CP}^{\,2}$.
We compare these constructions with the author's 2006 construction of minimal
symplectic four-manifolds having the integral cohomology ring of
$S^2\times S^2$, obtained via knot surgery and twisted fiber sums.
\end{abstract}

\maketitle

\section{Introduction}

The starting point of this paper is a simple observation about braided
surfaces in ruled four-manifolds.  Smith \cite{Smith} showed that, for $m\ge2$, the class $2m[\Sigma_g\times\{\mathrm{pt}\}]$ in $\Sigma_g\times S^2$ admits infinitely many pairwise non-isotopic connected symplectic representatives, while Hays--Park \cite{HaysPark} constructed infinite families of homotopic but pairwise non-isotopic symplectic surfaces in the product four-manifolds $\Sigma_g\times\Sigma_h$.  For the questions considered here, one can work
with a particularly elementary cyclic model.  We use these cyclic
multisections as the gluing surfaces in mixed symplectic fiber sums.

The relevant numerical feature is already visible in the genera of the multisections.  A connected
unbranched degree-$p$ multisection of
\[
 \Sigma_g\times S^2\longrightarrow\Sigma_g
\]
has genus $p(g-1)+1$.  Thus a degree-two multisection over $\Sigma_4$ and
a degree-three multisection over $\Sigma_3$ both have genus seven.  We
construct symplectic surfaces
\[
 F_2\subset\Sigma_4\times S^2,\qquad
 F_3\subset\Sigma_3\times S^2
\]
with
\[
 [F_2]=2[\Sigma_4\times\{\mathrm{pt}\}],\qquad
 [F_3]=3[\Sigma_3\times\{\mathrm{pt}\}],\qquad
 F_2^2=F_3^2=0.
\]
For an orientation-reversing identification of their normal-circle
boundaries we form
\[
 Z_\phi=
 (\Sigma_4\times S^2\setminus\operatorname{int}\nu F_2)
 \cup_\phi
 (\Sigma_3\times S^2\setminus\operatorname{int}\nu F_3).
\]
Every such sum satisfies
\[
 e(Z_\phi)=4,\qquad \sigma(Z_\phi)=0.
\]

\begin{theorem}\label{thm:main}
Up to interchanging the two ruled-product summands, the direct cyclic
mixed construction has the numerical invariants of $S^2\times S^2$ in
exactly three cases:
\[
 (p,q;g,h;G)=(2,3;4,3;7),\quad
 (2,4;3,2;5),\quad
 (3,3;2,2;4).
\]
In each case there are infinitely many adapted product-framed symplectic
gluings for which
\[
 H^*(Z;\mathbb Q)\cong H^*(S^2\times S^2;\mathbb Q)
\]
as graded rings, with finite first-homology groups of unbounded order.
Consequently the construction yields infinitely many pairwise
nondiffeomorphic symplectic rational cohomology
$S^2\times S^2$'s.  For every direct gluing, however, $\pi_1(Z)$ is
nontrivial and $H_1(Z;\mathbb Z)$ has quotients $\mathbb Z/p$ and
$\mathbb Z/q$.  In particular none of these direct sums has the integral
homology of $S^2\times S^2$.
\end{theorem}

The first question is whether one can make
$Z_\phi$ simply connected.  The answer is no: the boundary of the
complement of any connected unbranched multisection of degree at least
two is $\pi_1$-injective, so the two complement groups inject into the
amalgamated product $\pi_1(Z_\phi)$.

The situation for first homology is different.  Although the fundamental
group cannot be killed by the gluing, its free abelian part can be killed.
We compute
\[
 H_1(A;\mathbb Z)\cong\mathbb Z^8\oplus\mathbb Z/2,\qquad
 H_1(B;\mathbb Z)\cong\mathbb Z^6\oplus\mathbb Z/3
\]
for the two complements $A$ and $B$.  We also determine the image of
the boundary homology.  For every gluing, $H_1(Z_\phi;\mathbb Z)$ has
a quotient $\mathbb Z/2$ and a quotient $\mathbb Z/3$, so no direct sum
has the integral homology of $S^2\times S^2$.  On the other hand, by
choosing the genus-seven mapping class appropriately, the Mayer--Vietoris
map is an isomorphism over $\mathbb Q$ onto
\[
 H_1(A;\mathbb Q)\oplus H_1(B;\mathbb Q).
\]
Consequently $H_1(Z_\phi;\mathbb Q)=0$.  Together with
$e=4$ and $\sigma=0$, this gives the rational cohomology ring of
$S^2\times S^2$.

This also sits alongside earlier constructions of small exotic
four-manifolds obtained by combining symplectic sums, Luttinger surgeries
\cite{Luttinger}, and carefully chosen gluing maps
\cite{AkhmedovParkSmall,AkhmedovSmall,AkhmedovParkOdd}.  In those
constructions, Luttinger surgery is particularly useful for modifying the
fundamental group while retaining control of the symplectic structure.
The direct cyclic sums considered here behave differently: the
finite-index defect in the boundary homology survives every choice of
gluing map.

We compare this with the author's 2006 construction
\cite{AkhmedovCohomology} of minimal symplectic four-manifolds having the
\emph{integral} cohomology ring of $S^2\times S^2$, obtained via knot
surgery and twisted fiber sums.  This comparison shows that the finite quotients found here belong to the
direct unbranched multisection sum; they are not a general obstruction to
symplectic constructions with the cohomology of $S^2\times S^2$.  This
distinction is discussed again in Section~\ref{sec:comparison-akhmedov}.

The $(2,3)$ example is one of exactly three numerical possibilities in
the general mixed $(p,q)$ construction.  We classify the $e=4$ cases,
compute the general complement homology, and prove a rational
Mayer--Vietoris criterion.  The other two cases have degrees $(2,4)$ and
$(3,3)$ and gluing genera $5$ and $4$, respectively.

\section{Cyclic braided multisections}\label{sec:cyclic}

We construct the two surfaces needed here directly, without using a
low-degree specialization of the more elaborate non-isotopy constructions
of \cite{Smith,HaysPark}.

Let $S^1_{\mathrm{eq}}\subset S^2$ be a fixed equator.  Identify
$S^1_{\mathrm{eq}}$ with the unit complex numbers.

\subsection{A general construction}

Let $g\geq1$ and choose a primitive class
\[
 \xi\in H^1(\Sigma_g;\mathbb Z).
\]
Choose a smooth map
\[
 u:\Sigma_g\longrightarrow S^1
\]
representing $\xi$.  For an integer $p\geq2$, define
\begin{equation}\label{eq:Fgu}
 F_{g,p}(u)=
 \{(x,z)\in\Sigma_g\times S^1_{\mathrm{eq}}:\ z^p=u(x)\}.
\end{equation}

\begin{proposition}\label{prop:cyclicsurface}
The set $F_{g,p}(u)$ is a connected embedded symplectic surface in
$\Sigma_g\times S^2$.  The projection to the first factor is an unbranched
$p$-fold covering, and
\[
 [F_{g,p}(u)]=p[\Sigma_g\times\{\mathrm{pt}\}],\qquad
 F_{g,p}(u)^2=0,
\]
while
\[
 g(F_{g,p}(u))=p(g-1)+1.
\]
\end{proposition}

\begin{proof}
Consider the degree-$p$ covering
\[
 \rho_p:S^1\longrightarrow S^1,\qquad z\longmapsto z^p.
\]
The surface \eqref{eq:Fgu} is the pullback of $\rho_p$ by $u$:
\[
\begin{array}{ccc}
F_{g,p}(u) & \longrightarrow & S^1\\
{\scriptstyle q}\downarrow && \downarrow{\scriptstyle \rho_p}\\
\Sigma_g & \xrightarrow{\ u\ } & S^1 .
\end{array}
\]
Hence $q$ is an unbranched $p$-fold covering.  Since $\xi$ is primitive,
the induced homomorphism
\[
 u_*:\pi_1(\Sigma_g)\longrightarrow\pi_1(S^1)=\mathbb Z
\]
is onto.  Its reduction modulo $p$ is therefore onto $\mathbb Z/p$, so the
pullback cover is connected.

The defining equation in \eqref{eq:Fgu} also shows directly that
$F_{g,p}(u)$ is embedded in $\Sigma_g\times S^1_{\mathrm{eq}}$, and hence
in $\Sigma_g\times S^2$.

Let
\[
 B=[\Sigma_g\times\{\mathrm{pt}\}],\qquad
 V=[\{\mathrm{pt}\}\times S^2].
\]
Write $[F_{g,p}(u)]=aB+bV$.  A sphere fiber $V$ meets the multisection in
exactly $p$ positively oriented points, so
\[
 [F_{g,p}(u)]\cdot V=p.
\]
Since $B\cdot V=1$ and $V^2=0$, this gives $a=p$.

Now choose the representative of $B$ to be
$\Sigma_g\times\{N\}$, where $N$ is the north pole of $S^2$.  This surface
is disjoint from $F_{g,p}(u)$ because the latter lies in the equator.
Hence
\[
 [F_{g,p}(u)]\cdot B=0.
\]
Since $B^2=0$ and $V\cdot B=1$, this gives $b=0$.  Thus
\[
 [F_{g,p}(u)]=pB
\]
and consequently $F_{g,p}(u)^2=0$.

Riemann--Hurwitz for the unbranched cover $q$ gives
\[
 \chi(F_{g,p}(u))=p\chi(\Sigma_g)=p(2-2g),
\]
which is equivalent to
\[
 g(F_{g,p}(u))=p(g-1)+1.
\]

Finally take a product symplectic form
\[
 \omega=\lambda\,\omega_{\Sigma}+\omega_{S^2},
 \qquad \lambda>0.
\]
The second projection
$F_{g,p}(u)\to S^1_{\mathrm{eq}}\subset S^2$ has one-dimensional image,
so the pullback of the two-form $\omega_{S^2}$ to $F_{g,p}(u)$ is zero.
Therefore
\[
 \omega|_{F_{g,p}(u)}
 =\lambda\,q^*\omega_{\Sigma},
\]
which is a positive area form after orienting $F_{g,p}(u)$ by the covering
$q$.  Hence $F_{g,p}(u)$ is symplectic.
\end{proof}

\begin{remark}
The construction can equivalently be described by a cyclic spherical braid:
when $u(x)$ makes one positive turn around $S^1$, the $p$ roots in the
equator are cyclically permuted.  This is the simplest braid-monodromy model
needed here; Figure~\ref{fig:two-genus-seven} shows the two genus-seven cases used below.
\end{remark}

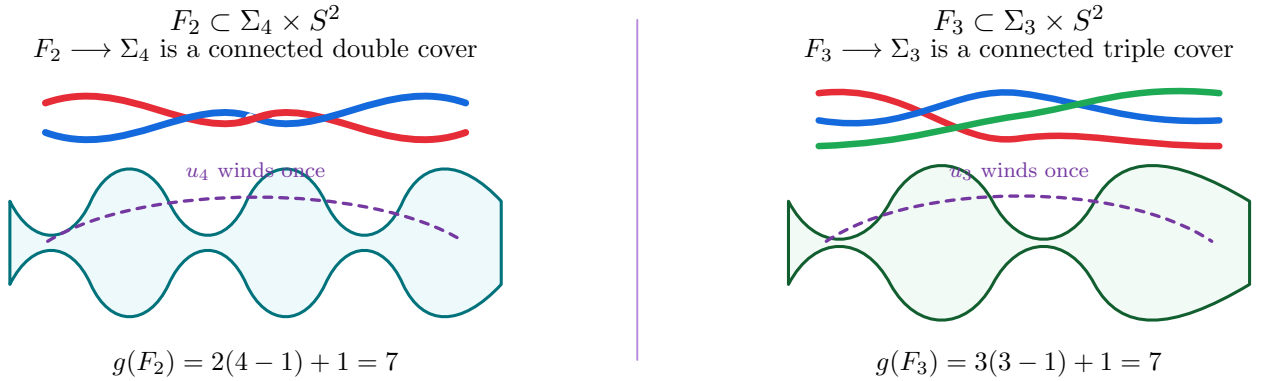
\begin{figure}[ht]
\centering
\begin{tikzpicture}[x=1cm,y=1cm,>=Latex,line cap=round,line join=round]

% ---------- left panel: F_2 ----------
\begin{scope}[xshift=-5.05cm]
  \node[font=\bfseries] at (0,2.35) {$F_2\subset \Sigma_4\times S^2$};
  \node[font=\small] at (0,1.95) {$F_2\longrightarrow\Sigma_4$ is a connected double cover};

  % stylized genus-4 base
  \draw[braidturquoise!65!black,line width=1.15pt,fill=braidturquoise!7]
    (-3.25,-1.15)
    .. controls (-2.95,-.55) and (-2.45,-.55) .. (-2.18,-1.15)
    .. controls (-1.92,-1.72) and (-1.42,-1.72) .. (-1.15,-1.15)
    .. controls (-.88,-.55) and (-.38,-.55) .. (-.12,-1.15)
    .. controls (.15,-1.72) and (.65,-1.72) .. (.92,-1.15)
    .. controls (1.19,-.55) and (1.69,-.55) .. (1.95,-1.15)
    .. controls (2.22,-1.72) and (2.72,-1.72) .. (3.25,-1.15)
    -- (3.25,-.05)
    .. controls (2.72,.52) and (2.22,.52) .. (1.95,-.05)
    .. controls (1.69,-.65) and (1.19,-.65) .. (.92,-.05)
    .. controls (.65,.52) and (.15,.52) .. (-.12,-.05)
    .. controls (-.38,-.65) and (-.88,-.65) .. (-1.15,-.05)
    .. controls (-1.42,.52) and (-1.92,.52) .. (-2.18,-.05)
    .. controls (-2.45,-.65) and (-2.95,-.65) .. (-3.25,-.05) -- cycle;

  % distinguished loop
  \draw[braidpurple!80!black,line width=1.25pt,dashed]
    (-2.75,-.58) .. controls (-1.5,.20) and (1.5,.20) .. (2.75,-.58);
  \node[font=\scriptsize,braidpurple!80!black] at (0,.35) {$u_4$ winds once};

  % braid strands lifted above loop
  \draw[braidred,line width=2.6pt]
    (-2.78,1.25) .. controls (-1.72,1.62) and (-.85,.72) .. (0,1.05)
    .. controls (.85,1.38) and (1.72,.48) .. (2.78,.86);
  \draw[braidblue,line width=2.6pt]
    (-2.78,.86) .. controls (-1.72,.48) and (-.85,1.38) .. (0,1.05)
    .. controls (.85,.72) and (1.72,1.62) .. (2.78,1.25);
  \fill[white] (-.07,1.05) circle (1.8pt);
  \draw[braidred,line width=2.6pt] (-.20,1.00) -- (.20,1.10);

  \node[font=\small,anchor=north] at (0,-1.88)
    {$g(F_2)=2(4-1)+1=7$};
\end{scope}

% ---------- right panel: F_3 ----------
\begin{scope}[xshift=5.05cm]
  \node[font=\bfseries] at (0,2.35) {$F_3\subset \Sigma_3\times S^2$};
  \node[font=\small] at (0,1.95) {$F_3\longrightarrow\Sigma_3$ is a connected triple cover};

  % stylized genus-3 base
  \draw[braidgreen!55!black,line width=1.15pt,fill=braidgreen!6]
    (-3.05,-1.15)
    .. controls (-2.70,-.48) and (-2.05,-.48) .. (-1.70,-1.15)
    .. controls (-1.35,-1.78) and (-.70,-1.78) .. (-.35,-1.15)
    .. controls (0,-.48) and (.65,-.48) .. (1.00,-1.15)
    .. controls (1.35,-1.78) and (2.00,-1.78) .. (3.05,-1.15)
    -- (3.05,-.05)
    .. controls (2.00,.58) and (1.35,.58) .. (1.00,-.05)
    .. controls (.65,-.72) and (0,-.72) .. (-.35,-.05)
    .. controls (-.70,.58) and (-1.35,.58) .. (-1.70,-.05)
    .. controls (-2.05,-.72) and (-2.70,-.72) .. (-3.05,-.05) -- cycle;

  \draw[braidpurple!80!black,line width=1.25pt,dashed]
    (-2.55,-.58) .. controls (-1.35,.22) and (1.35,.22) .. (2.55,-.58);
  \node[font=\scriptsize,braidpurple!80!black] at (0,.35) {$u_3$ winds once};

  % three-strand cyclic braid
  \draw[braidred,line width=2.45pt]
    (-2.65,1.38) .. controls (-1.45,1.52) and (-.85,.66) .. (0,.78)
    .. controls (.85,.90) and (1.45,.68) .. (2.65,.68);
  \draw[braidblue,line width=2.45pt]
    (-2.65,1.02) .. controls (-1.45,.82) and (-.85,1.48) .. (0,1.38)
    .. controls (.85,1.28) and (1.45,.96) .. (2.65,1.02);
  \draw[braidgreen,line width=2.45pt]
    (-2.65,.68) .. controls (-1.45,.72) and (-.85,.96) .. (0,1.08)
    .. controls (.85,1.20) and (1.45,1.50) .. (2.65,1.38);

  \node[font=\small,anchor=north] at (0,-1.88)
    {$g(F_3)=3(3-1)+1=7$};
\end{scope}

\draw[braidpurple!65,line width=.7pt] (0,-2.15)--(0,2.35);
\end{tikzpicture}
\caption{The two genus-seven cyclic multisections.  The lower silhouettes represent the base surfaces, with a distinguished loop on which the defining map $u_g$ has degree one.  Above that loop the sheets execute the indicated cyclic braid: two sheets for $F_2\to\Sigma_4$ and three for $F_3\to\Sigma_3$.  The pictures record the covering and braid monodromy; they are not meant as embeddings of the genus-seven surfaces in three-space.}
\label{fig:two-genus-seven}
\end{figure}

\subsection{The two genus-seven surfaces}

Apply Proposition~\ref{prop:cyclicsurface} with $(g,p)=(4,2)$ and
$(3,3)$.  We obtain
\[
 F_2:=F_{4,2}(u_4)\subset\Sigma_4\times S^2,
 \qquad
 F_3:=F_{3,3}(u_3)\subset\Sigma_3\times S^2,
\]
where $u_4,u_3$ represent primitive integral classes.

\begin{corollary}\label{cor:two}
The surfaces $F_2$ and $F_3$ are connected symplectic surfaces satisfying
\[
 g(F_2)=g(F_3)=7,
\]
\[
 [F_2]=2[\Sigma_4\times\{\mathrm{pt}\}],\qquad
 [F_3]=3[\Sigma_3\times\{\mathrm{pt}\}],
\]
and
\[
 F_2^2=F_3^2=0.
\]
\end{corollary}

\begin{figure}[ht]
\centering
\begin{tikzpicture}[x=1cm,y=1cm,>=Latex,line cap=round,line join=round]

% F2 fiber
\begin{scope}[xshift=-4.65cm]
  \node[font=\bfseries] at (0,2.15) {$F_2\cap(\{\mathrm{pt}\}\times S^2)$};
  % sphere
  \fill[braidturquoise!5] (0,0) circle (1.55);
  \draw[braidturquoise!65!black,line width=1.15pt] (0,0) circle (1.55);
  \draw[braidturquoise!45!black,dashed,line width=.8pt]
       (-1.55,0) arc[start angle=180,end angle=360,x radius=1.55,y radius=.42];
  \draw[braidturquoise!45!black,line width=.8pt]
       (-1.55,0) arc[start angle=180,end angle=0,x radius=1.55,y radius=.42];
  % incoming sheets
  \draw[braidred,line width=2.2pt] (-2.75,.75) .. controls (-2.0,.65) and (-1.45,.42) .. (-.58,.30);
  \draw[braidblue,line width=2.2pt] (-2.75,-.75) .. controls (-2.0,-.65) and (-1.45,-.42) .. (.58,-.30);
  \fill[braidred] (-.58,.30) circle (4.5pt);
  \fill[braidblue] (.58,-.30) circle (4.5pt);
  \node[font=\small] at (0,-1.95) {$F_2\cdot(\{\mathrm{pt}\}\times S^2)=2$};
  \node[font=\scriptsize,braidpurple!85!black] at (0,1.12) {two transverse positive points};
\end{scope}

% F3 fiber
\begin{scope}[xshift=4.65cm]
  \node[font=\bfseries] at (0,2.15) {$F_3\cap(\{\mathrm{pt}\}\times S^2)$};
  \fill[braidgreen!4] (0,0) circle (1.55);
  \draw[braidgreen!55!black,line width=1.15pt] (0,0) circle (1.55);
  \draw[braidgreen!40!black,dashed,line width=.8pt]
       (-1.55,0) arc[start angle=180,end angle=360,x radius=1.55,y radius=.42];
  \draw[braidgreen!40!black,line width=.8pt]
       (-1.55,0) arc[start angle=180,end angle=0,x radius=1.55,y radius=.42];
  \draw[braidred,line width=2.2pt] (-2.75,.88) .. controls (-2.0,.75) and (-1.45,.55) .. (-.62,.42);
  \draw[braidblue,line width=2.2pt] (-2.75,0) .. controls (-2.0,.03) and (-1.45,.04) .. (.62,.42);
  \draw[braidgreen,line width=2.2pt] (-2.75,-.88) .. controls (-2.0,-.72) and (-1.45,-.58) .. (0,-.58);
  \fill[braidred] (-.62,.42) circle (4.5pt);
  \fill[braidblue] (.62,.42) circle (4.5pt);
  \fill[braidgreen] (0,-.58) circle (4.5pt);
  \node[font=\small] at (0,-1.95) {$F_3\cdot(\{\mathrm{pt}\}\times S^2)=3$};
  \node[font=\scriptsize,braidpurple!85!black] at (0,1.12) {three transverse positive points};
\end{scope}

\draw[braidpurple!55,line width=.7pt] (0,-2.05)--(0,2.2);
\end{tikzpicture}
\caption{A fiberwise view of the two multisections.  A generic sphere fiber $\{\mathrm{pt}\}\times S^2$ meets $F_2$ in two points and $F_3$ in three points, corresponding to the two and three sheets of the respective coverings.  In particular,
$[F_2]=2[\Sigma_4\times\{\mathrm{pt}\}]$ and
$[F_3]=3[\Sigma_3\times\{\mathrm{pt}\}]$.}
\label{fig:fiber-intersections}
\end{figure}
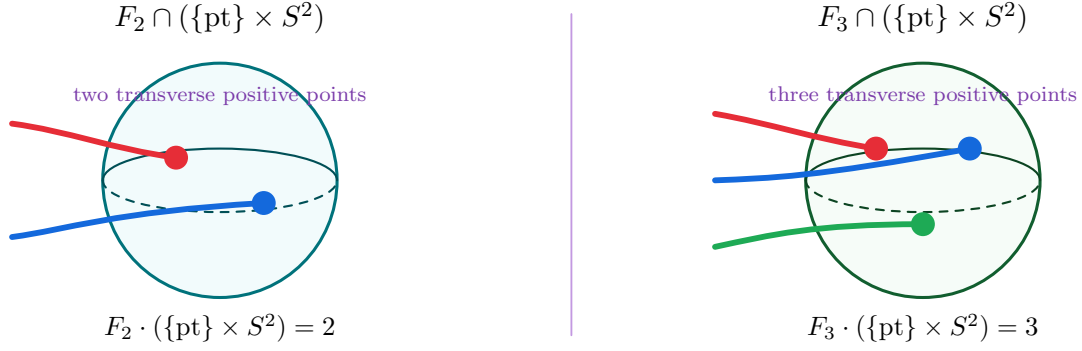

\section{The general mixed construction}\label{sec:general}

The pair $(g,p)=(4,2)$ and $(h,q)=(3,3)$ is one member of a more general
family.  Let
\[
 F_{g,p}\subset\Sigma_g\times S^2,\qquad
 F_{h,q}\subset\Sigma_h\times S^2
\]
be the cyclic multisections of Proposition~\ref{prop:cyclicsurface}.
They can be fiber-summed precisely when their genera agree, namely when
\begin{equation}\label{eq:genusmatch}
 p(g-1)=q(h-1)=:n.
\end{equation}
The common gluing surface then has genus
\[
 G=n+1.
\]
The geometry of the mixed sum is shown schematically in Figure~\ref{fig:mixed-sum}.
For a gluing diffeomorphism $\phi$ set
\[
 Z_{g,h}^{p,q}(\phi)=
 (\Sigma_g\times S^2\setminus\operatorname{int}\nu F_{g,p})
 \cup_\phi
 (\Sigma_h\times S^2\setminus\operatorname{int}\nu F_{h,q}).
\]

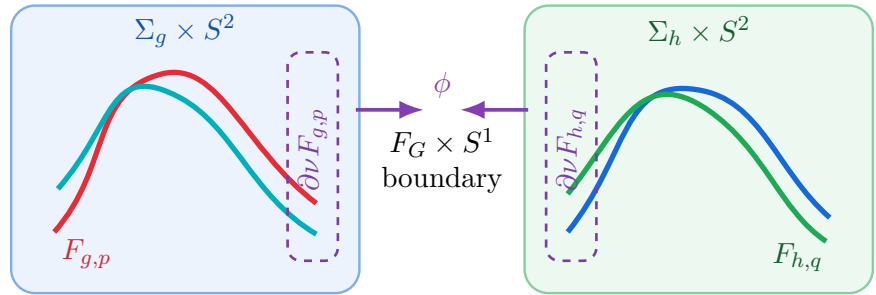
\begin{figure}[ht]
\centering
\begin{tikzpicture}[>=Latex,scale=0.95]
  % left block
  \fill[braidblue!8,rounded corners=9pt] (-6.0,-2.0) rectangle (-1.15,2.0);
  \draw[braidblue!55,line width=0.9pt,rounded corners=9pt] (-6.0,-2.0) rectangle (-1.15,2.0);
  \node[braidblue!80!black] at (-3.58,1.65) {$\Sigma_g\times S^2$};
  \draw[braidred,line width=2.0pt] (-5.4,-1.15) .. controls (-4.6,-0.25) and (-4.9,0.8) .. (-3.9,1.05)
     .. controls (-3.1,1.25) and (-2.75,0.0) .. (-1.75,-0.75);
  \draw[braidturquoise,line width=2.0pt] (-5.35,-0.55) .. controls (-4.55,0.35) and (-4.65,1.25) .. (-3.65,0.72)
     .. controls (-2.9,0.32) and (-2.65,-0.65) .. (-1.75,-1.18);
  \node[braidred!90!black] at (-4.95,-1.45) {$F_{g,p}$};
  \draw[braidpurple!85!black,dashed,line width=1.0pt,rounded corners=5pt] (-2.15,-1.55) rectangle (-1.45,1.35);
  \node[rotate=90,braidpurple!85!black] at (-1.8,0) {$\partial\nu F_{g,p}$};

  % right block
  \fill[braidgreen!8,rounded corners=9pt] (1.15,-2.0) rectangle (6.0,2.0);
  \draw[braidgreen!55,line width=0.9pt,rounded corners=9pt] (1.15,-2.0) rectangle (6.0,2.0);
  \node[braidgreen!55!black] at (3.58,1.65) {$\Sigma_h\times S^2$};
  \draw[braidblue,line width=2.0pt] (1.75,-1.15) .. controls (2.6,-0.1) and (2.5,1.0) .. (3.55,0.83)
     .. controls (4.45,0.7) and (4.65,-0.3) .. (5.4,-0.95);
  \draw[braidgreen,line width=2.0pt] (1.75,-0.62) .. controls (2.45,0.25) and (2.75,1.2) .. (3.7,0.55)
     .. controls (4.5,0.0) and (4.7,-0.85) .. (5.35,-1.28);
  \node[braidgreen!55!black] at (4.95,-1.5) {$F_{h,q}$};
  \draw[braidpurple!85!black,dashed,line width=1.0pt,rounded corners=5pt] (1.45,-1.55) rectangle (2.15,1.35);
  \node[rotate=90,braidpurple!85!black] at (1.8,0) {$\partial\nu F_{h,q}$};

  % gluing arrows and center
  \draw[->,braidpurple!90!black,line width=1.4pt] (-1.2,0.55) -- (-0.25,0.55);
  \draw[->,braidpurple!90!black,line width=1.4pt] (1.2,0.55) -- (0.25,0.55);
  \node[braidpurple!90!black] at (0,0.9) {$\phi$};
  \node[align=center] at (0,-0.15) {$F_G\times S^1$\\boundary};
\end{tikzpicture}
\caption{Schematic of the mixed fiber sum.  The cyclic multisections have the same genus $G$; after removing tubular neighborhoods, their boundaries are identified by a diffeomorphism $\phi$ reversing the normal-circle orientation.  The colored strands are only a visual reminder of the multisection structure.}
\label{fig:mixed-sum}
\end{figure}

\begin{proposition}[General geography]\label{prop:general-geography}
Assume \eqref{eq:genusmatch}.  Then
\[
 e\bigl(Z_{g,h}^{p,q}(\phi)\bigr)
 =8-4g-4h+4n
\]
and
\[
 \sigma\bigl(Z_{g,h}^{p,q}(\phi)\bigr)=0.
\]
Equivalently,
\[
 e=4n\left(1-\frac1p-\frac1q\right).
\]
\end{proposition}

\begin{proof}
Since
\[
 e(\Sigma_g\times S^2)=4-4g,\qquad
 e(\Sigma_h\times S^2)=4-4h
\]
and $e(\Sigma_G)=2-2G=-2n$, the fiber-sum formula gives
\[
 e=(4-4g)+(4-4h)-2(-2n)
 =8-4g-4h+4n.
\]
Using $g=1+n/p$ and $h=1+n/q$ yields
\[
 e=4n\left(1-\frac1p-\frac1q\right).
\]
Both product summands have signature zero, hence so does the sum.
\end{proof}

\begin{corollary}[The $e=4$ equation]\label{cor:e4general}
Under \eqref{eq:genusmatch}, one has $e=4$ if and only if
\begin{equation}\label{eq:diophantine}
 n\left(1-\frac1p-\frac1q\right)=1,
\end{equation}
or equivalently
\begin{equation}\label{eq:triangle}
 \frac1p+\frac1q+\frac1n=1.
\end{equation}
\end{corollary}

\begin{theorem}[Classification of the numerical cases]
\label{thm:classification}
Assume $p,q\ge2$, $g,h\ge2$, and
$p(g-1)=q(h-1)=n$.  If the mixed sum has $e=4$, then, up to interchanging
the two summands,
\[
 (p,q,n)=(2,3,6),\qquad(2,4,4),\qquad(3,3,3).
\]
The corresponding triples $(g,h,G)$ are
\[
 (4,3,7),\qquad(3,2,5),\qquad(2,2,4).
\]
\end{theorem}

\begin{proof}
Equation \eqref{eq:triangle} is elementary to classify.  Reorder its
three denominators as $2\le r\le s\le t$.  Since
$1=1/r+1/s+1/t\le3/r$, one has $r\le3$.

If $r=3$, then $1/s+1/t=2/3$.  Since $s,t\ge3$, equality forces
$s=t=3$.  If $r=2$, then $1/s+1/t=1/2$, so $s\le4$.  The cases
$s=3$ and $s=4$ give $t=6$ and $t=4$, respectively.  Thus, up to
permutation, the only solutions are
\[
 (2,3,6),\qquad(2,4,4),\qquad(3,3,3).
\]
Finally, $n$ is distinguished by the covering requirements
$p\mid n$ and $q\mid n$.  These leave, up to exchanging $p$ and $q$,
exactly the three displayed choices.  The formulas
$g=1+n/p$, $h=1+n/q$, and $G=n+1$ give the stated genera.
\end{proof}

\begin{corollary}\label{cor:threefamilies}
The three numerical cases are
\[
 (p,q;g,h;G)=(2,3;4,3;7),\quad
 (2,4;3,2;5),\quad
 (3,3;2,2;4).
\]
\end{corollary}

\begin{proposition}\label{prop:general-H1}
For the cyclic degree-$p$ multisection,
\[
 H_1(M_{g,p};\mathbb Z)
 \cong\mathbb Z^{2g}\oplus\mathbb Z/p.
\]
For a mixed $(p,q)$ sum, every gluing has quotients
\[
 H_1\bigl(Z_{g,h}^{p,q}(\phi);\mathbb Z\bigr)
 \twoheadrightarrow\mathbb Z/p,\qquad
 H_1\bigl(Z_{g,h}^{p,q}(\phi);\mathbb Z\bigr)
 \twoheadrightarrow\mathbb Z/q.
\]
In particular no direct mixed sum is an integral homology
$S^2\times S^2$ whenever $p,q\ge2$.
\end{proposition}

\begin{proof}
The first assertion is the general version of the calculations in
Section~\ref{sec:complements}: the cyclic monodromy identifies all
peripheral classes in the abelianization and their product relation gives
$p\mu=0$.  The second assertion follows from
Lemmas~\ref{lem:image-cover} and~\ref{lem:boundaryH1}: on each side the
boundary surface maps to an index-$p$ or index-$q$ subgroup of the base
homology, and the corresponding quotient survives the Mayer--Vietoris
cokernel.
\end{proof}

For a controlled class of framings, the integral statement can be sharpened.  The north-pole section gives a preferred splitting
\[
 H_1(M_{g,p};\mathbb Z)
 \cong H_1(\Sigma_g;\mathbb Z)\oplus\mathbb Z/p.
\]

\begin{lemma}\label{lem:adapted-framing}
There is a framing of the trivial normal bundle of $F_{g,p}$ for which
the push-off of every class in $H_1(F_{g,p};\mathbb Z)$ has zero
component in the meridional summand $\mathbb Z/p$ of
$H_1(M_{g,p};\mathbb Z)$.
\end{lemma}

\begin{proof}
Fix any framing.  Relative to the displayed splitting, the meridional
component of the push-off map is a homomorphism
\[
 \eta:H_1(F_{g,p};\mathbb Z)\longrightarrow\mathbb Z/p.
\]
The normal bundle is trivial, and its homotopy classes of framings form
an affine space over
\[
 [F_{g,p},S^1]\cong H^1(F_{g,p};\mathbb Z).
\]
Changing the framing by
$\alpha\in H^1(F_{g,p};\mathbb Z)$ changes the push-off of
$x\in H_1(F_{g,p};\mathbb Z)$ by $\langle\alpha,x\rangle$ meridians,
and hence changes $\eta$ by the reduction of $\alpha$ modulo $p$.
Since $H_1(F_{g,p};\mathbb Z)$ is free abelian,
the reduction map
\[
 H^1(F_{g,p};\mathbb Z)\longrightarrow
 \operatorname{Hom}(H_1(F_{g,p};\mathbb Z),\mathbb Z/p)
\]
is surjective.  Choose $\alpha$ reducing to $-\eta$.  In the resulting
framing the meridional component vanishes identically.
\end{proof}

We call such a framing \emph{adapted}.  By an adapted product-framed
gluing below we mean a boundary identification
\[
 \phi(x,e^{it})=(f(x),e^{-it})
\]
with respect to adapted framings on the two sides and with no additional
rim twist.

\begin{proposition}
\label{prop:productH1}
Assume the genus-matching condition and let
\[
 \Theta_f=
 \begin{pmatrix}
 q_{p*}\\
 -q_{q*}f_*
 \end{pmatrix}
 :
 H_1(\Sigma_G;\mathbb Z)
 \longrightarrow
 H_1(\Sigma_g;\mathbb Z)\oplus H_1(\Sigma_h;\mathbb Z).
\]
For an adapted product-framed gluing one has
\[
 H_1(Z_{g,h}^{p,q}(\phi);\mathbb Z)
 \cong
 \operatorname{coker}\Theta_f
 \oplus
 \frac{\mathbb Z/p\oplus\mathbb Z/q}
 {\langle(1,-1)\rangle}.
\]
Consequently
\[
 H_1(Z_{g,h}^{p,q}(\phi);\mathbb Z)
 \cong
 \operatorname{coker}\Theta_f\oplus\mathbb Z/\gcd(p,q).
\]
If the first homology is finite, then its order is divisible by $pq$.
In the square case $G=g+h$, finiteness is equivalent to
$\det\Theta_f\ne0$, and in that case
\[
 \left|H_1(Z_{g,h}^{p,q}(\phi);\mathbb Z)\right|
 =
 \gcd(p,q)\,|\det\Theta_f|.
\]
\end{proposition}

\begin{proof}
The restriction to the adapted product framings of
Lemma~\ref{lem:adapted-framing} is essential here, since arbitrary rim
twists can mix surface and meridional components.  With the chosen product framings,
\[
 H_1(M_{g,p};\mathbb Z)
 =
 H_1(\Sigma_g;\mathbb Z)\oplus\mathbb Z/p
\]
and similarly on the other side.  Surface push-offs have no meridional
component, while the normal circle maps to the meridional generator.
Thus the Mayer--Vietoris map splits into the surface map $\Theta_f$ and
the meridional map
\[
 \mathbb Z\longrightarrow\mathbb Z/p\oplus\mathbb Z/q,
 \qquad
 1\longmapsto(1,-1).
\]
Taking cokernels gives the asserted direct sum.  The second cokernel is
cyclic of order $\gcd(p,q)$.

If $\operatorname{coker}\Theta_f$ is finite, projecting it to either
base and then quotienting by the covering image gives surjections onto
$\mathbb Z/p$ and $\mathbb Z/q$.  Hence its order is divisible by
$\operatorname{lcm}(p,q)$, and the order of the full first homology is
therefore divisible by
$\gcd(p,q)\operatorname{lcm}(p,q)=pq$.  When $G=g+h$, the map
$\Theta_f$ is square, and if its cokernel is finite then
$|\operatorname{coker}\Theta_f|=|\det\Theta_f|$, which gives the
stated formula.
\end{proof}

\begin{lemma}\label{lem:symp-general-position}
Let $V$ be a symplectic vector space of dimension $2G$ and let
$K,L\subset V$ be symplectic subspaces.  There is
$T\in\operatorname{Sp}(V)$ such that
\[
 \dim(K\cap T^{-1}L)
 =
 \max\{0,\dim K+\dim L-2G\}.
\]
\end{lemma}

\begin{proof}
Write $\dim K=2a$ and $\dim L=2b$.  Choose symplectic bases of $K$ and
$L$ and extend them to symplectic bases of $V$.  If $a+b\le G$, send
$K$ to a symplectic coordinate subspace contained in the symplectic
complement of $L$; then the intersection is zero.  If $a+b>G$, split
$K$ as a symplectic direct sum of dimensions
$2(a+b-G)$ and $2(G-b)$, map the first summand onto a symplectic
subspace of $L$ of dimension $2(a+b-G)$, and map the second into
$L^\perp$.  Extending these identifications on symplectic complements
gives a symplectic automorphism with intersection dimension
$2(a+b-G)=\dim K+\dim L-2G$.
\end{proof}

\begin{theorem}[General rational-homology criterion]\label{thm:general-rational}
Assume \eqref{eq:genusmatch}.  There exists a gluing $\phi$ for which
\[
 H_1\bigl(Z_{g,h}^{p,q}(\phi);\mathbb Q\bigr)=0
\]
if and only if
\begin{equation}\label{eq:dimensioncriterion}
 G\ge g+h.
\end{equation}
When equality holds, the rational Mayer--Vietoris surface map can be
chosen to be an isomorphism.
\end{theorem}

\begin{proof}
Over $\mathbb Q$ the meridional torsion disappears.  Write
\[
 V=H_1(\Sigma_G;\mathbb Q),\qquad
 A=q_{p*}:V\to H_1(\Sigma_g;\mathbb Q),\qquad
 B=q_{q*}:V\to H_1(\Sigma_h;\mathbb Q).
\]
The Mayer--Vietoris map on the surface factor is
\[
 \Theta_T=(A,-BT),
 \qquad T\in\operatorname{Sp}(V).
\]
Its target has dimension $2g+2h$, while $\dim V=2G$, so
$G\ge g+h$ is necessary for surjectivity.

Put $K_A=\ker A$ and $K_B=\ker B$.  By the same argument as in
Lemma~\ref{lem:symplectic-kernel}, both are symplectic subspaces, of
dimensions
\[
 \dim K_A=2G-2g,\qquad
 \dim K_B=2G-2h.
\]
By Lemma~\ref{lem:symp-general-position} one may arrange
\[
 \dim(K_A\cap T^{-1}K_B)
 =
 \max\{0,\dim K_A+\dim K_B-2G\}
 =
 \max\{0,2G-2g-2h\}.
\]
When $G\ge g+h$, this minimum equals $2G-2g-2h$.  Therefore
\[
 \operatorname{rank}\Theta_T
 =2G-\dim(K_A\cap T^{-1}K_B)
 =2g+2h,
\]
so $\Theta_T$ is surjective.  Notice that the intersection is zero only
in the equality case $G=g+h$; this is why the equality case gives an
isomorphism rather than merely a surjection.

The maximal-rank condition is the nonvanishing of suitable minors and
hence defines a nonempty Zariski-open subset of the symplectic group.
The arithmetic subgroup $\operatorname{Sp}(2G,\mathbb Z)$ is Zariski
dense in the ambient symplectic group \cite{BorelDensity}; consequently
the open set contains an integral symplectic matrix.  Finally the
surjection
\[
 \operatorname{Mod}(\Sigma_G)\twoheadrightarrow
 \operatorname{Sp}(2G,\mathbb Z)
\]
\cite{FarbMargalit} realizes that matrix by a surface diffeomorphism.
\end{proof}

\begin{corollary}[Rational cohomology $S^2\times S^2$ cases]
\label{cor:classifiedQ}
For each of the three cases in Theorem~\ref{thm:classification}, there is
a gluing for which
\[
 H^*(Z;\mathbb Q)\cong H^*(S^2\times S^2;\mathbb Q)
\]
as graded rings.  For every direct gluing in these cases,
$\pi_1(Z)\ne1$ and $H_1(Z;\mathbb Z)\ne0$.
\end{corollary}

\begin{proof}
The equation $e=4$ gives
\[
 g+h=2+\frac np+\frac nq=n+1=G.
\]
Thus Theorem~\ref{thm:general-rational} applies with equality and yields
$b_1(Z)=0$ over $\mathbb Q$.  Since $e(Z)=4$ and $\sigma(Z)=0$,
Poincare duality gives $b_2=2$ and $b_2^+=b_2^-=1$.  By Poincare duality and the universal coefficient theorem, the
intersection pairing on $H_2(Z;\mathbb Z)/\operatorname{Tor}$ is
unimodular.  It has rank two and signature zero, so it is either the even
hyperbolic form $H$ or the odd form
$\langle1\rangle\oplus\langle-1\rangle$.  Both become hyperbolic over
$\mathbb Q$.  Since $H^1(Z;\mathbb Q)=H^3(Z;\mathbb Q)=0$, this
identifies the full graded rational cohomology ring with that of
$S^2\times S^2$.  The fundamental-group and integral
homology assertions follow from Theorem~\ref{thm:nogo} and
Proposition~\ref{prop:general-H1}.
\end{proof}

\begin{theorem}[Infinitely many first homology groups and diffeomorphism types]
\label{thm:infinitefamilies}
Fix any one of the three numerical cases in
Theorem~\ref{thm:classification}.  There is a sequence of
adapted product-framed symplectic gluings $\{\phi_k\}$ such that
\[
 H^*(Z_{\phi_k};\mathbb Q)
 \cong
 H^*(S^2\times S^2;\mathbb Q)
\]
for every $k$, while the finite groups
$H_1(Z_{\phi_k};\mathbb Z)$ have unbounded order.  In particular, the
groups $\pi_1(Z_{\phi_k})$ are pairwise nonisomorphic after passing to a
subsequence, and the manifolds $Z_{\phi_k}$ are pairwise
nondiffeomorphic.
\end{theorem}

\begin{proof}
Write
\[
 A=q_{p*},\qquad B=q_{q*},
\]
and, for $T\in\operatorname{Sp}(2G,\mathbb Z)$, set
\[
 \Delta(T)=
 \det
 \begin{pmatrix}
 A\\
 BT
 \end{pmatrix}.
\]
Because $G=g+h$ in the three $e=4$ cases, this is the determinant of a
square $2G\times2G$ matrix.  The rational transversality argument in
Theorem~\ref{thm:general-rational} shows that $\Delta$ is not identically
zero on the symplectic group.

For the integral argument,
Theorem~\ref{thm:general-rational} gives
$T_+\in\operatorname{Sp}(2G,\mathbb Z)$ with $\Delta(T_+)\ne0$.
On the other hand, the integral kernels of $A$ and $B$ are primitive
sublattices: the quotient by either kernel identifies with the image of
the corresponding covering push-forward, which is a free abelian
group.  Thus both kernels contain primitive integral vectors.  Choose
primitive $x\in\ker A$ and $y\in\ker B$.  The integral symplectic group
acts transitively on primitive vectors, so there is
$T_-\in\operatorname{Sp}(2G,\mathbb Z)$ with $T_-x=y$.  Then
\[
 x\in\ker A\cap T_-^{-1}\ker B,
\]
so $\Delta(T_-)=0$.

Set
\[
 S=T_+T_-^{-1}\in\operatorname{Sp}(2G,\mathbb Z).
\]
The mapping class group surjects onto
$\operatorname{Sp}(2G,\mathbb Z)$ and is generated by Dehn twists, whose
actions on first homology are symplectic transvections
\cite{FarbMargalit}.  Choose a mapping class lifting $S$ and factor it
as a product of Dehn twists.  Applying the corresponding transvections
successively to $T_-$ gives a finite sequence beginning at $T_-$ and
ending at $T_+$.  Since $\Delta(T_-)=0$ and $\Delta(T_+)\ne0$, at some
adjacent step, after interchanging the two matrices in that pair if
necessary, there are an integral symplectic matrix $T_0$ and a primitive
integral vector $v$ for which
\[
 \Delta(T_0)\ne\Delta(\tau_vT_0).
\]
It follows that the restriction of $\Delta$ to
\[
 T_k=\tau_v^kT_0,\qquad k\in\mathbb Z,
\]
is nonconstant.  Here
\[
 \tau_v(x)=x+\langle x,v\rangle v
\]
is the symplectic transvection associated to $v$.  Since
$\tau_v^k-I$ has rank one and depends linearly on $k$, the matrix
\[
 \begin{pmatrix}A\\BT_k\end{pmatrix}
\]
is a rank-one linear perturbation of the matrix for $T_0$.  By
multilinearity of the determinant, a rank-one perturbation contributes
at most one perturbed column to a nonzero determinant term; hence
\[
 \Delta(T_k)=a k+b
\]
for some integers $a,b$.  The inequality
$\Delta(T_0)\ne\Delta(T_1)$ gives $a\ne0$.

For all but at most one integer $k$, $\Delta(T_k)\ne0$.  Hence the
corresponding gluing has $H_1(-;\mathbb Q)=0$, and the geography
calculation gives the rational cohomology ring of
$S^2\times S^2$.  By Proposition~\ref{prop:productH1},
\[
 |H_1(Z_{\phi_k};\mathbb Z)|
 =
 \gcd(p,q)\,|ak+b|,
\]
which is unbounded.

The mapping class group surjects onto
$\operatorname{Sp}(2G,\mathbb Z)$ \cite{FarbMargalit}, so the $T_k$ are
realized by surface diffeomorphisms.  After isotopy, these can be taken area preserving, and
Gompf's symplectic sum construction therefore gives symplectic
$Z_{\phi_k}$.  Finally, diffeomorphic manifolds have isomorphic first
homology, and isomorphic fundamental groups have isomorphic
abelianizations.  Passing to a subsequence with pairwise distinct orders
of $H_1$ proves both assertions.
\end{proof}

\begin{remark}
The theorem gives an elementary way of distinguishing infinitely many
members of the braided family: no Seiberg--Witten calculation is needed.
The distinction is already visible in the integral first homology.  This
does not address the more delicate problem of producing infinitely many
smooth structures with a fixed integral homology group.
\end{remark}

\section{The basic $(2,3)$ example}\label{sec:geography}

Choose an orientation-reversing bundle diffeomorphism
\[
 \phi:\partial\nu F_2\longrightarrow\partial\nu F_3
\]
and set
\[
 Z_\phi=
 (\Sigma_4\times S^2\setminus\Int\nu F_2)
 \cup_\phi
 (\Sigma_3\times S^2\setminus\Int\nu F_3).
\]

\begin{proposition}\label{prop:geography}
For every choice of $\phi$,
\[
 e(Z_\phi)=4,\qquad \sigma(Z_\phi)=0.
\]
\end{proposition}

\begin{proof}
We have
\[
 e(\Sigma_g\times S^2)=(2-2g)\cdot2=4-4g.
\]
Thus
\[
 e(\Sigma_4\times S^2)=-12,\qquad
 e(\Sigma_3\times S^2)=-8.
\]
Since $e(\Sigma_7)=2-14=-12$, the fiber-sum formula gives
\[
 e(Z_\phi)
 =-12-8-2(-12)=4.
\]
Both products have signature zero.  Novikov additivity therefore gives
\[
 \sigma(Z_\phi)=0.
\]
\end{proof}

\begin{remark}
If $Z_\phi$ were simply connected, then $b_2(Z_\phi)=2$ and
$b_2^+=b_2^-=1$.  One would still need to determine parity: the even
intersection form is $H$, while the odd form is
$\langle1\rangle\oplus\langle-1\rangle$.  The next sections show that the
simple-connectivity hypothesis already fails for the direct multisection
sum.
\end{remark}

\section{Complement groups}\label{sec:complements}

We compute the complements of the cyclic multisections using a bundle
description that will also be used for the incompressibility theorem in
Section~\ref{sec:obstruction}.

Let
\[
 M_{g,p}=(\Sigma_g\times S^2)\setminus\Int\nu F_{g,p}(u).
\]
We choose the tubular neighborhood $\nu F_{g,p}(u)$ fiberwise, so that
its intersection with each sphere fiber is a union of $p$ small disks
centered at the points of the multisection.  For each $x\in\Sigma_g$,
the fiber of the projection
$M_{g,p}\to\Sigma_g$ is
\[
 P_p=S^2\setminus\bigcup_{j=1}^p\Int D_j,
\]
a sphere with $p$ open disks removed.  The cyclic motion of the roots gives
the monodromy.

Choose standard generators
\[
 \pi_1(\Sigma_g)=
 \left\langle
 a_1,b_1,\ldots,a_g,b_g\ \middle|\
 \prod_{i=1}^g[a_i,b_i]=1
 \right\rangle
\]
so that
\[
 u_*(a_1)=1,\qquad
 u_*(b_1)=u_*(a_i)=u_*(b_i)=0\quad(i\ge2).
\]
Let $r:P_p\to P_p$ be the rotation that cyclically permutes the $p$
boundary components.

The north pole $N$ is fixed by this rotation and is disjoint from all
punctures.  Hence
\[
 x\longmapsto(x,N)
\]
is a section of $M_{g,p}\to\Sigma_g$.  It follows that the fundamental
group extension splits.

\begin{proposition}\label{prop:generalpi1}
There is a split exact sequence
\[
 1\longrightarrow\pi_1(P_p)
 \longrightarrow\pi_1(M_{g,p})
 \longrightarrow\pi_1(\Sigma_g)
 \longrightarrow1,
\]
and
\[
 \pi_1(M_{g,p})
 \cong
 \pi_1(P_p)\rtimes\pi_1(\Sigma_g),
\]
where $a_1$ acts by $r_*$ and all the other standard base generators act
trivially on $\pi_1(P_p)$.
\end{proposition}

\begin{proof}
The homotopy exact sequence of the fiber bundle gives the extension, because
$\pi_2(\Sigma_g)=0$.  The north-pole section splits it.  By construction the
configuration of punctures rotates once by $2\pi/p$ when $u$ winds once
around $S^1$, and is stationary for loops on which $u_*$ vanishes.  This is
exactly the stated action.
\end{proof}

\subsection{The degree-two complement}

For $p=2$, $P_2$ is an annulus.  Let $\mu$ generate
$\pi_1(P_2)\cong\mathbb Z$.  The rotation exchanging the two boundary
components reverses the core orientation, so
\[
 r_*(\mu)=\mu^{-1}.
\]

\begin{corollary}\label{cor:Agroup}
For
\[
 A=(\Sigma_4\times S^2)\setminus\Int\nu F_2
\]
one has
\[
\begin{aligned}
 \pi_1(A)=
 \langle\,&a_1,b_1,\ldots,a_4,b_4,\mu\mid\\
 &\prod_{i=1}^4[a_i,b_i]=1,\quad
 a_1\mu a_1^{-1}=\mu^{-1},\\
 &[b_1,\mu]=[a_i,\mu]=[b_i,\mu]=1,\quad i=2,3,4
 \rangle .
\end{aligned}
\]
Moreover
\[
 H_1(A;\mathbb Z)\cong\mathbb Z^8\oplus\mathbb Z/2.
\]
\end{corollary}

\begin{proof}
Only the homology calculation requires comment.  After abelianization the
relation $a_1\mu a_1^{-1}=\mu^{-1}$ becomes $2\mu=0$, and no further
relation is imposed on the eight base generators.
\end{proof}

\subsection{The degree-three complement}

For $p=3$, let $\mu_1,\mu_2,\mu_3$ be positively oriented peripheral
loops in the pair of pants $P_3$, chosen so that
\[
 \mu_1\mu_2\mu_3=1.
\]
Choose the basepoint on the rotation axis.  The rotation can then be
represented so that
\[
 r_*(\mu_1)=\mu_2,\qquad
 r_*(\mu_2)=\mu_3,\qquad
 r_*(\mu_3)=\mu_1.
\]

\begin{corollary}\label{cor:Bgroup}
For
\[
 B=(\Sigma_3\times S^2)\setminus\Int\nu F_3
\]
one has the presentation
\[
\begin{aligned}
 \pi_1(B)=\langle\,
 &c_1,d_1,c_2,d_2,c_3,d_3,\mu_1,\mu_2,\mu_3\mid\\
 &[c_1,d_1][c_2,d_2][c_3,d_3]=1,\quad
 \mu_1\mu_2\mu_3=1,\\
 &c_1\mu_1c_1^{-1}=\mu_2,\quad
 c_1\mu_2c_1^{-1}=\mu_3,\quad
 c_1\mu_3c_1^{-1}=\mu_1,\\
 &[x,\mu_j]=1
 \quad\text{for }x\in\{d_1,c_2,d_2,c_3,d_3\},\ j=1,2,3
 \,\rangle .
\end{aligned}
\]
Furthermore
\[
 H_1(B;\mathbb Z)\cong\mathbb Z^6\oplus\mathbb Z/3.
\]
\end{corollary}

\begin{proof}
The presentation follows from Proposition~\ref{prop:generalpi1}.  In the
abelianization the three peripheral classes become equal:
\[
 [\mu_1]=[\mu_2]=[\mu_3]=:\mu.
\]
The pair-of-pants relation becomes $3\mu=0$.  The six base generators
remain free.
\end{proof}

\begin{remark}
The torsion summands $\mathbb Z/2$ and $\mathbb Z/3$ in the two
abelianizations are useful consistency checks, but they do not imply that
the meridians have finite order in the nonabelian complement groups.  In
fact the fiber groups inject in both cases, so the peripheral elements have
infinite order.
\end{remark}

\section{Boundary homology and the Mayer--Vietoris map}\label{sec:mv}

Let $F=F_{g,p}(u)$ be one of the cyclic multisections from
Section~\ref{sec:cyclic}, and let
\[
 M=M_{g,p}=(\Sigma_g\times S^2)\setminus\operatorname{int}\nu F.
\]
Since $F^2=0$, after choosing a normal framing we identify
$\partial\nu F\cong F\times S^1$.  Let $m$ denote the normal-circle class.

The projection $M\to\Sigma_g$ induces a split epimorphism on first
homology.  The fiber contribution is the cyclic meridional torsion
computed in Section~\ref{sec:complements}, and hence
\[
 H_1(M;\mathbb Z)\cong H_1(\Sigma_g;\mathbb Z)\oplus\mathbb Z/p.
\]
Let $q:F\to\Sigma_g$ be the $p$-fold covering.

\begin{lemma}\label{lem:image-cover}
For the cyclic cover determined by the primitive class
$\xi=[u]\in H^1(\Sigma_g;\mathbb Z)$,
\[
 q_*H_1(F;\mathbb Z)
 =
 \ker\!\left(
 H_1(\Sigma_g;\mathbb Z)
 \xrightarrow{\ \langle\xi,\cdot\rangle\ }
 \mathbb Z\longrightarrow\mathbb Z/p
 \right).
\]
In particular,
\[
 H_1(\Sigma_g;\mathbb Z)/q_*H_1(F;\mathbb Z)\cong\mathbb Z/p.
\]
\end{lemma}

\begin{proof}
Choose standard generators so that $u_*(a_1)=1$ and $u_*$ vanishes on
$b_1,a_2,b_2,\ldots,a_g,b_g$.  The covering subgroup contains
$a_1^p$ and all of these latter generators.  Hence the image of $H_1(F)$
contains
\[
 p[a_1],\ [b_1],\ [a_2],\ [b_2],\ldots,[a_g],[b_g].
\]
Conversely, every loop in the covering subgroup has $u_*$ divisible by
$p$, so its homology class lies in the displayed kernel.  The two
subgroups are therefore equal.
\end{proof}

\begin{lemma}\label{lem:boundaryH1}
Under the projection
\[
 H_1(M;\mathbb Z)\longrightarrow H_1(\Sigma_g;\mathbb Z),
\]
the image of $H_1(\partial\nu F;\mathbb Z)$ is exactly
$q_*H_1(F;\mathbb Z)$.  The normal circle maps to the meridional class
in the $\mathbb Z/p$ summand.
\end{lemma}

\begin{proof}
A loop in the $F$ factor of $F\times S^1$ projects to its image under
$q$, while the normal circle projects trivially to the base.  A change of
normal framing may add a multiple of the meridian to a surface push-off,
but cannot alter its projection to $H_1(\Sigma_g)$.  The normal circle is
a boundary component of the punctured-sphere fiber and represents the
meridional class.
\end{proof}

\begin{proposition}\label{prop:integral-obstruction}
For every gluing diffeomorphism
$\phi:\partial\nu F_2\to\partial\nu F_3$, the group
$H_1(Z_\phi;\mathbb Z)$ admits surjections
\[
 H_1(Z_\phi;\mathbb Z)\twoheadrightarrow\mathbb Z/2,\qquad
 H_1(Z_\phi;\mathbb Z)\twoheadrightarrow\mathbb Z/3.
\]
Consequently no direct mixed fiber sum has the integral homology of
$S^2\times S^2$.
\end{proposition}

\begin{proof}
Mayer--Vietoris gives
\[
 H_1(\partial\nu F_2)
 \xrightarrow{\Psi_\phi}
 H_1(A)\oplus H_1(B)
 \longrightarrow H_1(Z_\phi)\longrightarrow0.
\]
Project the first factor to $H_1(\Sigma_4)$ and then quotient by
$q_{2*}H_1(F_2)$.  Lemma~\ref{lem:boundaryH1} shows that this composite
vanishes on $\operatorname{im}\Psi_\phi$, independently of $\phi$, and
therefore descends to a surjection
\[
 H_1(Z_\phi)\twoheadrightarrow
 H_1(\Sigma_4)/q_{2*}H_1(F_2)\cong\mathbb Z/2.
\]
The same argument on the second factor gives the quotient $\mathbb Z/3$.
\end{proof}

The free part can nevertheless be killed.  Put
$V=H_1(\Sigma_7;\mathbb Q)$.  After an initial identification of the two
genus-seven surfaces, let
\[
 Q_2:V\to H_1(\Sigma_4;\mathbb Q),\qquad
 Q_3:V\to H_1(\Sigma_3;\mathbb Q)
\]
be the covering push-forwards, and let $K_i=\ker Q_i$.  Then
$\dim K_2=6$ and $\dim K_3=8$.

\begin{lemma}\label{lem:symplectic-kernel}
Each $K_i$ is a symplectic subspace of $V$.
\end{lemma}

\begin{proof}
For a finite oriented covering $q:F\to\Sigma_g$ of degree $p$,
$q^*H^1(\Sigma_g;\mathbb Q)$ is nondegenerate for the intersection
pairing because
\[
 \langle q^*\alpha,q^*\beta\rangle_F
 =p\,\langle\alpha,\beta\rangle_{\Sigma_g}.
\]
The push-forward $q_*$ is adjoint to $q^*$ under Poincare duality, so
$\ker q_*=(\operatorname{im}q^*)^\perp$.  The symplectic orthogonal
complement of a nondegenerate symplectic subspace is symplectic.
\end{proof}

\begin{lemma}\label{lem:linear-gluing}
There exists $T\in\operatorname{Sp}(14,\mathbb Z)$ such that
\[
 K_2\cap T^{-1}K_3=\{0\}
\]
over $\mathbb Q$.
\end{lemma}

\begin{proof}
The symplectic spaces $K_2$ and $K_3^\perp$ both have dimension six.
Choose a symplectic isomorphism $K_2\to K_3^\perp$ and extend it to a
symplectic automorphism of $V$.  For this automorphism the required
intersection is zero.  The transversality condition is the nonvanishing
of a determinant, hence defines a nonempty Zariski-open subset of
$\operatorname{Sp}(14,\mathbb Q)$.  By Borel density \cite{BorelDensity},
$\operatorname{Sp}(14,\mathbb Z)$ is Zariski dense in the ambient
symplectic group, so the open set contains an integral symplectic
matrix.
\end{proof}

\begin{proposition}[Rational homology gluing]\label{prop:rational-gluing}
There is a gluing diffeomorphism $\phi$ for which
\[
 H_1(Z_\phi;\mathbb Q)=0.
\]
\end{proposition}

\begin{proof}
The mapping class group of $\Sigma_7$ surjects onto
$\operatorname{Sp}(14,\mathbb Z)$, so choose a surface diffeomorphism
inducing the matrix $T$ from Lemma~\ref{lem:linear-gluing}.  Over
$\mathbb Q$ the meridional torsion vanishes, and the Mayer--Vietoris map
on the surface factor becomes
\[
 \Theta_T:V\longrightarrow
 H_1(\Sigma_4;\mathbb Q)\oplus H_1(\Sigma_3;\mathbb Q),
 \qquad
 x\longmapsto(Q_2x,-Q_3Tx).
\]
Its kernel is $K_2\cap T^{-1}K_3=0$.  Both sides have dimension $14$,
so $\Theta_T$ is an isomorphism.  Thus $H_1(Z_\phi;\mathbb Q)=0$.
\end{proof}

\begin{theorem}[Rational cohomology $S^2\times S^2$]\label{thm:Qcohomology}
There are symplectic fiber sums $Z_\phi$ obtained from the two explicit
genus-seven braided multisections such that
\[
 H^*(Z_\phi;\mathbb Q)\cong
 H^*(S^2\times S^2;\mathbb Q)
\]
as graded rings.  For every such $Z_\phi$ one nevertheless has
$\pi_1(Z_\phi)\neq1$ and $H_1(Z_\phi;\mathbb Z)\neq0$.
\end{theorem}

\begin{proof}
Choose $\phi$ as in Proposition~\ref{prop:rational-gluing}.  Then
$b_1=0$, and by Poincare duality $b_3=0$.  Since $e(Z_\phi)=4$,
we get $b_2=2$.  The signature is zero, so
$b_2^+=b_2^-=1$.

By Poincare duality together with the universal coefficient theorem,
the intersection pairing on
$H_2(Z_\phi;\mathbb Z)/\operatorname{Tor}$ is unimodular.  It has rank
two and signature zero, hence is either
\[
 H=\begin{pmatrix}0&1\\1&0\end{pmatrix}
 \quad\text{or}\quad
 \langle1\rangle\oplus\langle-1\rangle.
\]
Both forms are hyperbolic over $\mathbb Q$.  Since $H^1(Z_\phi;\mathbb
Q)=H^3(Z_\phi;\mathbb Q)=0$, this determines the rational cohomology
ring and gives
\[
 H^*(Z_\phi;\mathbb Q)\cong H^*(S^2\times S^2;\mathbb Q).
\]

The nontriviality of $\pi_1$ follows from
Theorem~\ref{thm:nogo}, while Proposition~\ref{prop:integral-obstruction}
gives $H_1(Z_\phi;\mathbb Z)\neq0$.

Finally, rescale the product symplectic forms so that $F_2$ and $F_3$
have equal area.  Any chosen orientation-preserving mapping class can be
represented by an area-preserving diffeomorphism after an isotopy, by
Moser's theorem.  Gompf's symplectic sum construction therefore gives a
symplectic form on $Z_\phi$.
\end{proof}

\begin{remark}
If $H_1(Z_\phi)$ is finite, Proposition~\ref{prop:integral-obstruction}
forces its order to be divisible by both $2$ and $3$, hence by $6$.
Determining the smallest possible integral first homology is an integral
Smith-normal-form problem for the Mayer--Vietoris matrix and is not claimed
here.
\end{remark}

\section{Boundary incompressibility}\label{sec:obstruction}

The needed point is the injectivity of the boundary group.

\begin{theorem}\label{thm:incompressible}
Let
\[
 F\subset\Sigma_g\times S^2,\qquad g\ge1,
\]
be a connected unbranched multisection of degree $p\ge2$.  Put
\[
 M=(\Sigma_g\times S^2)\setminus\Int\nu F.
\]
Then the inclusion
\[
 \iota:\partial\nu F\longrightarrow M
\]
induces an injection
\[
 \iota_*:\pi_1(\partial\nu F)\hookrightarrow\pi_1(M).
\]
\end{theorem}

\begin{proof}
Choose $\nu F$ fiberwise: over each $x\in\Sigma_g$ it consists
of disjoint small disks about the $p$ points of
$F\cap(\{x\}\times S^2)$.  Since the multisection is unbranched, these
disks vary locally trivially with $x$.  The projection
\[
 M\longrightarrow\Sigma_g
\]
is therefore a bundle with fiber
\[
 P_p=S^2\setminus\bigcup_{j=1}^p\Int D_j.
\]
The restriction of the projection to the boundary can be viewed as follows.
The boundary $\partial\nu F$ is a circle bundle over $F$, and
$q:F\to\Sigma_g$ is the original $p$-fold covering.

Choose a basepoint on one boundary circle in a fiber.  Since
$g\ge1$, both $F$ and $\Sigma_g$ are aspherical in degree two:
$\pi_2(F)=\pi_2(\Sigma_g)=0$.  The homotopy exact sequences of the two
bundles therefore give the following commutative diagram with exact
rows:
\[
\begin{array}{ccccccccc}
1&\longrightarrow&\pi_1(S^1)&\longrightarrow&
\pi_1(\partial\nu F)&\longrightarrow&\pi_1(F)&\longrightarrow&1\\
&&\downarrow&&\downarrow{\scriptstyle\iota_*}&&
\downarrow{\scriptstyle q_*}&&\\
1&\longrightarrow&\pi_1(P_p)&\longrightarrow&
\pi_1(M)&\longrightarrow&\pi_1(\Sigma_g)&\longrightarrow&1 .
\end{array}
\]
We justify the two injective outer vertical maps.

First, $q_*$ is injective because $q:F\to\Sigma_g$ is a covering.
Second, a boundary circle of $P_p$ is $\pi_1$-injective when $p\ge2$:
for $p=2$ the fiber is an annulus, and for $p\ge3$ it is a compact surface
of negative Euler characteristic with incompressible boundary.

Now let $\gamma\in\pi_1(\partial\nu F)$ satisfy
$\iota_*(\gamma)=1$.  Its image in $\pi_1(F)$ maps trivially under $q_*$.
Since $q_*$ is injective, the image of $\gamma$ in $\pi_1(F)$ is trivial.
Thus $\gamma$ lies in the subgroup generated by the circle fiber of
$\partial\nu F\to F$.  But the left vertical map
$\pi_1(S^1)\to\pi_1(P_p)$ is injective, and therefore $\gamma=1$.
The theorem follows.
\end{proof}

\begin{remark}
The proof uses only that the surface is a connected unbranched
multisection of degree at least two.  It does not depend on the cyclic
model of Section~\ref{sec:cyclic}, nor on any particular braid word.
\end{remark}

\section{The obstruction to the direct fiber sum}

We apply Theorem~\ref{thm:incompressible} to an arbitrary gluing.

\begin{lemma}\label{lem:normalform}
Let $C\to G_1$ and $C\to G_2$ be injective homomorphisms.  Then the natural
homomorphisms
\[
 G_1\longrightarrow G_1*_C G_2,\qquad
 G_2\longrightarrow G_1*_C G_2
\]
are injective.
\end{lemma}

\begin{proof}
This is the standard normal-form theorem for a free product with
amalgamation.  A reduced word of length one represented by a nontrivial
element of either factor cannot represent the identity in the amalgam.
\end{proof}

\begin{theorem}\label{thm:nogo}
For $i=1,2$, let
\[
 F_i\subset\Sigma_{g_i}\times S^2
\]
be connected unbranched multisections of degrees $p_i\ge2$, with $g_i\ge1$,
and suppose their normal bundles admit a boundary identification
\[
 \phi:\partial\nu F_1\longrightarrow\partial\nu F_2.
\]
Then
\[
 Z=
 (\Sigma_{g_1}\times S^2\setminus\Int\nu F_1)
 \cup_\phi
 (\Sigma_{g_2}\times S^2\setminus\Int\nu F_2)
\]
is not simply connected.

More precisely, the fundamental group of each complement injects into
$\pi_1(Z)$.
\end{theorem}

\begin{proof}
Let
\[
 M_i=(\Sigma_{g_i}\times S^2)\setminus\Int\nu F_i.
\]
Van Kampen identifies
\[
 \pi_1(Z)
 \cong
 \pi_1(M_1)
 *_{\pi_1(\partial\nu F_1)}
 \pi_1(M_2),
\]
where the second boundary map is composed with $\phi_*$.  By
Theorem~\ref{thm:incompressible}, both boundary homomorphisms are
injective.  Lemma~\ref{lem:normalform} therefore implies that
\[
 \pi_1(M_i)\hookrightarrow\pi_1(Z)
\]
for $i=1,2$.

Finally, $\pi_1(M_i)$ is nontrivial: the projection
$M_i\to\Sigma_{g_i}$ induces a surjection on fundamental groups, and
$\pi_1(\Sigma_{g_i})\neq1$ because $g_i\ge1$.  Hence $\pi_1(Z)\neq1$.
\end{proof}

\begin{corollary}\label{cor:mixednogo}
For the genus-seven surfaces
\[
 F_2\subset\Sigma_4\times S^2,\qquad
 F_3\subset\Sigma_3\times S^2
\]
of Corollary~\ref{cor:two}, no choice of gluing diffeomorphism
\[
 \phi:\partial\nu F_2\to\partial\nu F_3
\]
makes the mixed fiber sum $Z_\phi$ simply connected.
\end{corollary}

\begin{proof}
Apply Theorem~\ref{thm:nogo} with $(g_1,p_1)=(4,2)$ and
$(g_2,p_2)=(3,3)$.
\end{proof}

\begin{remark}
The genus-seven boundary has fourteen surface generators, exactly the sum
$8+6$ of the numbers of standard generators of
$\pi_1(\Sigma_4)$ and $\pi_1(\Sigma_3)$.  This numerical coincidence
suggests that a gluing map might identify the generators in a way that
kills both groups.  Theorem~\ref{thm:nogo} shows why this cannot happen:
the relations imposed by the gluing form an amalgamation along an
\emph{injective} boundary subgroup.  They identify subgroups of the two
complement groups; they do not quotient either factor by those subgroups.
\end{remark}

\begin{remark}
The abelianizations
\[
 H_1(A)\cong\mathbb Z^8\oplus\mathbb Z/2,\qquad
 H_1(B)\cong\mathbb Z^6\oplus\mathbb Z/3
\]
make the coprimality of $2$ and $3$ look promising.  However, the
meridians themselves have infinite order in $\pi_1(A)$ and $\pi_1(B)$.
The $\mathbb Z/2$ and $\mathbb Z/3$ occur only after abelianization.
Thus one cannot infer nonabelian relations $\mu^2=1$ and $\mu^3=1$ in the
fiber-sum group.
\end{remark}

\section{Comparison with the integral construction and further topology}
\label{sec:comparison-akhmedov}

The use of symplectic sums together with surgeries to control the topology
of small four-manifolds also appears in
\cite{AkhmedovParkSmall,AkhmedovSmall,AkhmedovParkOdd}.  Here we focus on
the closer comparison with the author's 2006 construction
\cite{AkhmedovCohomology} of minimal symplectic four-manifolds having the
integral cohomology ring of $S^2\times S^2$, obtained via knot surgery and
twisted fiber sums.  Both constructions use symplectic sums, but the
gluing surfaces play quite different roles in first homology.

Let $K$ be a genus-one fibered knot and let $M_K$ be the three-manifold
obtained by $0$-framed surgery on $K$.  The product $M_K\times S^1$ is a
symplectic torus bundle over a torus.  In \cite{AkhmedovCohomology}, a
twisted fiber sum of two copies is first formed by identifying a torus
fiber of one copy with a torus section of the other.  The resulting
symplectic manifold $Y_K$ contains a square-zero genus-two symplectic
surface
\[
 \Sigma_2=T_1\#T_2.
\]
Two copies of $Y_K$ are then summed along $\Sigma_2$.

The decisive point is the second gluing.  In the notation of
\cite{AkhmedovCohomology}, the first homology of the genus-two gluing
surface is generated by
\[
 m,\ x,\ \gamma_1,\ \gamma_2,
\]
and the gluing interchanges the two types of generators:
\[
 m'\longmapsto\gamma_1,\qquad
 \gamma_1'\longmapsto m,\qquad
 x'\longmapsto\gamma_2,\qquad
 \gamma_2'\longmapsto x.
\]
The classes $m,x$ are homologically essential in the relevant building
block, whereas $\gamma_1,\gamma_2$ are homologically trivial there.  The
Mayer--Vietoris map therefore kills the first homology of the final sum,
and the resulting manifold $X_K$ satisfies
\[
 H_1(X_K;\mathbb Z)=0,\qquad
 H_2(X_K;\mathbb Z)\cong\mathbb Z^2,\qquad
 Q_{X_K}\cong
 H=\begin{pmatrix}0&1\\1&0\end{pmatrix}.
\]
Thus $X_K$ has the integral cohomology ring of $S^2\times S^2$; it also
has $e=4$, $\sigma=0$, and is minimal symplectic.

The cyclic multisection construction has a different integral feature.
For
\[
 F_{g,p}\longrightarrow\Sigma_g,
\]
Lemma~\ref{lem:image-cover} gives
\[
 H_1(\Sigma_g;\mathbb Z)/
 q_*H_1(F_{g,p};\mathbb Z)\cong\mathbb Z/p.
\]
Hence the boundary image never contains all of the integral base
homology.  In a mixed $(p,q)$ sum, the corresponding defects survive as
quotients $\mathbb Z/p$ and $\mathbb Z/q$ of the final first homology.
No mapping class of the common gluing surface removes this defect,
because it is already detected after projection to the two base homology
groups.  This is precisely why the direct construction naturally gives
rational, rather than integral, cohomology $S^2\times S^2$'s.

The comparison also shows where a modification would have to occur.
Changing only the boundary diffeomorphism cannot remove the integral
quotients.  One must first alter the complements so that suitable boundary
surface generators become null-homologous, as the $\gamma_i$ do in
\cite{AkhmedovCohomology}.  Luttinger surgeries on Lagrangian tori in the
complements are natural candidates: they can change $H_1$ and $\pi_1$
without changing Euler characteristic or signature.  If one also seeks a
simply connected manifold, the resulting nonabelian fundamental group
must of course be computed separately.

\begin{remark}[Twisted bundle case]\label{rem:twisted-ruled}
The same idea can be considered in the nontrivial oriented $S^2$-bundle
over a surface.  Write $S$ for a section of odd self-intersection and
$F$ for the fiber.  Then
\[
 S\cdot F=1,\qquad F^2=0,\qquad S^2\equiv1\pmod2.
\]
A degree-$p$ multisection in the class $pS+kF$ has square
\[
 p^2S^2+2pk.
\]
Consequently a square-zero class of this form can occur only when $p$ is
even, the first essential difference from the product case.

There is a more concrete odd-form counterpart of the geography above.
Normalize the section so that $S^2=1$.  A degree-$p$ multisection in the
class
\[
 A=pS+kF
\]
has square $A^2=p^2+2pk$.  Thus a square-zero class of this form can occur
only when $p$ is even, in which case
\[
 A_{g,p}=pS-\frac p2F.
\]
For the nontrivial ruled surface over $\Sigma_g$ the canonical class in
this basis is
\[
 K=-2S+(2g-1)F.
\]
Hence
\[
 K\cdot A_{g,p}=2p(g-1),
\]
and adjunction gives
\[
 g(A_{g,p})=p(g-1)+1.
\]
The genus formula is therefore the same as for the unbranched cyclic
multisections in the product case, although the parity restriction is
different.

Among the three degree pairs in Theorem~\ref{thm:classification}, only
$(2,4)$ has both degrees even.  This singles out the following candidate.
Let $E_3\to\Sigma_3$ and $E_2\to\Sigma_2$ be the nontrivial oriented
$S^2$-bundles, with sections $S_3,S_2$ of square $+1$ and sphere fibers
$F_3,F_2$.  The square-zero classes
\[
 A_{3,2}=2S_3-F_3,\qquad A_{2,4}=4S_2-2F_2
\]
both have genus $5$, and Proposition~\ref{prop:twistedexistence} gives
explicit connected unbranched symplectic representatives.  We may
therefore form the mixed symplectic sum
\[
 (E_3\setminus\Int\nu A_{3,2})
 \mathop{\#}_{\phi}
 (E_2\setminus\Int\nu A_{2,4}).
\]
The Euler characteristic and signature calculation is the same as in the
product $(2,4)$ case and gives
\[
 e=4,\qquad \sigma=0.
\]

Smith observed that the braid construction of symplectic multisections also
admits a twisted-sphere-bundle version \cite{Smith}.  For the two classes
needed here, one can give a direct finite-holonomy model.

\begin{proposition}\label{prop:twistedexistence}
Let $E_g\to\Sigma_g$ denote the nontrivial oriented $S^2$-bundle.

\begin{enumerate}
\item The bundle $E_3$ contains a connected unbranched symplectic
degree-$2$ multisection in the class
\[
        2S_3-F_3.
\]
\item The bundle $E_2$ contains a connected unbranched symplectic
degree-$4$ multisection in the class
\[
        4S_2-2F_2.
\]
\end{enumerate}
Both surfaces have self-intersection zero and genus $5$.
\end{proposition}

\begin{proof}
We use flat $S^2$-bundles with holonomy in $\operatorname{SO}(3)$.  Write
$R_x(\theta)$ and $R_z(\theta)$ for rotations through angle $\theta$
about the $x$- and $z$-axes.  The obstruction to lifting an
$\operatorname{SO}(3)$ representation of a surface group to
$\operatorname{SU}(2)$ is its second Stiefel--Whitney class.

For the degree-$2$ case, choose generators
\[
 \pi_1(\Sigma_3)=
 \langle a_1,b_1,a_2,b_2,a_3,b_3\mid
 [a_1,b_1][a_2,b_2][a_3,b_3]=1\rangle
\]
and define
\[
 \rho_2(a_1)=R_x(\pi),\qquad
 \rho_2(b_1)=R_z(\pi),
\]
with all remaining generators mapped to the identity.  The two rotations
commute in $\operatorname{SO}(3)$, so this is a representation.  If they
are lifted to the unit quaternions, their commutator is $-1$; hence the
associated $S^2$-bundle has nonzero $w_2$ and is therefore the nontrivial
bundle $E_3$.

Let $P_2=\{N,S\}\subset S^2$ be the north and south poles.  The holonomy
preserves $P_2$: $R_z(\pi)$ fixes both points and $R_x(\pi)$ interchanges
them.  Thus
\[
 \widetilde\Sigma_3\times_{\rho_2}P_2
 \subset
 \widetilde\Sigma_3\times_{\rho_2}S^2=E_3
\]
is a connected unbranched degree-$2$ multisection.

For the degree-$4$ case, take
\[
 \pi_1(\Sigma_2)=
 \langle a_1,b_1,a_2,b_2\mid[a_1,b_1][a_2,b_2]=1\rangle
\]
and set
\[
 \rho_4(a_1)=R_x(\pi),\qquad
 \rho_4(b_1)=R_z(\pi),\qquad
 \rho_4(a_2)=R_z(\pi/2),\qquad
 \rho_4(b_2)=1.
\]
Again the surface relation is satisfied.  The first pair of holonomies
has lifts whose commutator is $-1$, while the second commutator is $1$,
so the associated bundle has $w_2\ne0$ and is $E_2$.

Let $P_4$ be the four equally spaced points on the equator.  The image of
$\rho_4$ preserves $P_4$, and $R_z(\pi/2)$ acts transitively on it.
Therefore
\[
 \widetilde\Sigma_2\times_{\rho_4}P_4
 \subset E_2
\]
is a connected unbranched degree-$4$ multisection.

The rotations preserve the standard area form on $S^2$.  Hence on each
flat associated bundle the form
\[
 \omega_\lambda=\lambda\,\pi^*\omega_{\Sigma_g}+\omega_{S^2},
 \qquad \lambda>0,
\]
descends from $\widetilde\Sigma_g\times S^2$ and is symplectic.  On the
multisections above the vertical term vanishes, so the restriction is
$\lambda\,q^*\omega_{\Sigma_g}>0$.  Thus both multisections are
symplectic.

Their normal bundles have finite holonomy, so their Euler classes are
torsion.  Since the second integral cohomology of each multisection is
torsion-free, the Euler classes vanish; in particular both surfaces have
self-intersection zero.  If $S_g^2=1$ and
$F_g$ is the sphere fiber, a degree-$p$ class has the form
$pS_g+kF_g$.  The square-zero condition gives
\[
 p^2+2pk=0,
\]
and therefore $k=-p/2$.  The two classes are consequently
\[
 2S_3-F_3,\qquad 4S_2-2F_2.
\]
Finally, since the coverings are unbranched,
\[
 g= p(g(\Sigma)-1)+1,
\]
which gives genus $5$ in both cases.
\end{proof}

\begin{proposition}\label{prop:twisted24homology}
Let $E_g\to\Sigma_g$ be the nontrivial oriented $S^2$-bundle, with a
section $S$ satisfying $S^2=1$ and fiber $F$.  Let $A_{g,p}$ be a
connected unbranched degree-$p$ square-zero multisection in the class
\[
 [A_{g,p}]=pS-\frac p2F,
 \qquad p\ {\rm even},
\]
and let
\[
 M_{g,p}^{\rm tw}=E_g\setminus\Int\nu A_{g,p}.
\]
Then
\[
 H_1(M_{g,p}^{\rm tw};\mathbb Z)
 \cong \mathbb Z^{2g}\oplus\mathbb Z/(p/2).
\]
In particular,
\[
 H_1(M_{3,2}^{\rm tw};\mathbb Z)\cong\mathbb Z^6,
 \qquad
 H_1(M_{2,4}^{\rm tw};\mathbb Z)\cong
 \mathbb Z^4\oplus\mathbb Z/2.
\]

For the twisted $(2,4)$ mixed sum, there are boundary identifications for
which
\[
 H_1(Z_{\phi}^{\rm tw};\mathbb Q)=0.
\]
For every such gluing,
\[
 b_1(Z_{\phi}^{\rm tw})=0,\qquad
 b_2(Z_{\phi}^{\rm tw})=2,\qquad
 \sigma(Z_{\phi}^{\rm tw})=0,
\]
and hence
\[
 H^*(Z_{\phi}^{\rm tw};\mathbb Q)
 \cong
 H^*(\mathbb CP^2\#\overline{\mathbb CP}^{\,2};\mathbb Q)
\]
as graded rings.
\end{proposition}

\begin{proof}
The homology calculation is most transparent from the pair
$(E_g,M_{g,p}^{\rm tw})$.  By excision and the Thom isomorphism,
\[
 H_2(E_g,M_{g,p}^{\rm tw};\mathbb Z)
 \cong H_0(A_{g,p};\mathbb Z)\cong\mathbb Z,
 \qquad
 H_1(E_g,M_{g,p}^{\rm tw};\mathbb Z)=0.
\]
The relevant part of the long exact sequence is therefore
\[
 H_2(E_g;\mathbb Z)\longrightarrow\mathbb Z
 \longrightarrow H_1(M_{g,p}^{\rm tw};\mathbb Z)
 \longrightarrow H_1(E_g;\mathbb Z)\longrightarrow0.
\]
The first map is intersection with $[A_{g,p}]$.  Since
$H_2(E_g;\mathbb Z)$ is generated by $S$ and $F$,
\[
 F\cdot A_{g,p}=p,\qquad
 S\cdot A_{g,p}=\frac p2.
\]
Its image is therefore $(p/2)\mathbb Z$.  Since
$H_1(E_g;\mathbb Z)\cong H_1(\Sigma_g;\mathbb Z)$ is free, the resulting
short exact sequence splits as a sequence of abelian groups, giving
\[
 H_1(M_{g,p}^{\rm tw};\mathbb Z)
 \cong H_1(\Sigma_g;\mathbb Z)\oplus\mathbb Z/(p/2).
\]

For the rational gluing statement, the torsion terms disappear.  Choose
the tubular neighborhoods fiberwise.  The projection of the complement to
the base then has punctured-sphere fibers, and the projection of the
boundary surface factor is precisely the covering map
$A_{g,p}\to\Sigma_g$.  Since the meridian maps to torsion in
$H_1(M_{g,p}^{\rm tw};\mathbb Z)$, it vanishes over $\mathbb Q$.
Thus over $\mathbb Q$ the relevant boundary map is governed by the
covering push-forward.  In the $(2,4)$ case we have
\[
 H_1(A_{3,2};\mathbb Q)\cong H_1(\Sigma_5;\mathbb Q)
 \cong H_1(A_{2,4};\mathbb Q),
\]
and the two covering maps give surjections
\[
 Q_2:H_1(\Sigma_5;\mathbb Q)\longrightarrow H_1(\Sigma_3;\mathbb Q),
 \qquad
 Q_4:H_1(\Sigma_5;\mathbb Q)\longrightarrow H_1(\Sigma_2;\mathbb Q).
\]
Their kernels have dimensions $4$ and $6$, respectively.  As in the
product case, these kernels are symplectic subspaces.  Since
\[
 4+6=10=\dim H_1(\Sigma_5;\mathbb Q),
\]
we may choose $T\in\operatorname{Sp}(10,\mathbb Z)$ so that
\[
 T(\ker Q_2)\cap\ker Q_4=0.
\]
Realizing $T$ by a mapping class of $\Sigma_5$, the rational
Mayer--Vietoris map
\[
 x\longmapsto (Q_2x,-Q_4Tx)
\]
is an isomorphism.  Hence
$H_1(Z_{\phi}^{\rm tw};\mathbb Q)=0$.

The Euler characteristic and signature are unchanged from the product
$(2,4)$ case:
\[
 e(Z_{\phi}^{\rm tw})=4,\qquad
 \sigma(Z_{\phi}^{\rm tw})=0.
\]
Thus $b_2=2$ and $b_2^+=b_2^-=1$.  The intersection pairing on
$H_2(Z_{\phi}^{\rm tw};\mathbb Z)/\operatorname{Tor}$ is an integral
unimodular form of rank two and signature zero.  Hence it is either the
even hyperbolic form $H$ or the odd form
$\langle1\rangle\oplus\langle-1\rangle$.  These two forms become
isomorphic over $\mathbb Q$.  It follows that
\[
 H^*(Z_{\phi}^{\rm tw};\mathbb Q)
 \cong
 H^*(\mathbb CP^2\#\overline{\mathbb CP}^{\,2};\mathbb Q)
\]
as graded rings.

After rescaling the symplectic forms, the two genus-five multisections may
be assumed to have the same symplectic area.  The chosen gluing mapping
class may be represented, after isotopy, by an area-preserving
diffeomorphism.  Gompf's symplectic sum construction therefore equips
$Z_{\phi}^{\rm tw}$ with a symplectic form.
\end{proof}

For a gluing as in Proposition~\ref{prop:twisted24homology} one has
$b_1=0$ and $b_2=2$.  There is also a
useful mod-$2$ distinction between the two complements.  For the
nontrivial ruled surface,
\[
 w_2(E_g)=\operatorname{PD}(F)\pmod2.
\]
On the degree-$2$ side,
\[
 A_{3,2}=2S_3-F_3\equiv F_3\pmod2,
\]
so $\operatorname{PD}(A_{3,2})=w_2(E_3)$.  Thus removal of
$A_{3,2}$ may remove the spin obstruction.  On the degree-$4$ side,
however,
\[
 A_{2,4}=4S_2-2F_2\equiv0\pmod2.
\]
Let
\[
 M_4=E_2\setminus\Int\nu A_{2,4}.
\]
By the cohomology sequence of the pair $(E_2,M_4)$ and the Thom
isomorphism, the image of
\[
 H^2(E_2,M_4;\mathbb Z/2)\longrightarrow H^2(E_2;\mathbb Z/2)
\]
is generated by $\operatorname{PD}(A_{2,4})$.  Thus
\[
 \ker\bigl(H^2(E_2;\mathbb Z/2)\longrightarrow
 H^2(M_4;\mathbb Z/2)\bigr)
 =\langle\operatorname{PD}(A_{2,4})\rangle.
\]
Since $A_{2,4}\equiv0\pmod2$, this kernel is zero, and therefore the
restriction map is injective.  Since
$w_2(E_2)=\operatorname{PD}(F_2)\ne0$, it follows that
\[
 w_2(M_4)=w_2(E_2)|_{M_4}\ne0.
\]
In particular, the degree-$4$ complement is non-spin.  Consequently
every mixed sum obtained by gluing this complement to the degree-$2$
complement is non-spin: a spin structure on the closed sum would restrict
to a spin structure on $M_4$.

This proves non-spinness independently of the gluing map.  One should
distinguish this from oddness of the free integral intersection form when
$H_1$ has $2$-torsion.  If the resulting $w_2$ evaluates nontrivially on
the free part of $H_2$---in particular, if $H_1=0$---then the rank-two,
signature-zero free intersection form is odd and therefore
\[
 \langle1\rangle\oplus\langle-1\rangle,
\]
the form of $\mathbb CP^2\#\overline{\mathbb CP}^{\,2}$.  Thus the
twisted $(2,4)$ construction gives a concrete route toward symplectic
four-manifolds having the rational cohomology ring of
$\mathbb CP^2\#\overline{\mathbb CP}^{\,2}$, while the integral parity
question reduces to an explicit calculation of the surviving
Stiefel--Whitney class and the torsion in first homology.

This completes the existence and rational-homology part of the twisted
$(2,4)$ construction.  We do not pursue the full integral classification
here.  The precise finite groups
arising from the integral Mayer--Vietoris map, the isotopy theory of the
multisections, and the parity of the free integral intersection form will
be studied in follow-up work in \cite{AkhmedovTorus}.  The non-spin calculation above is
independent of those remaining integral questions.
\end{remark}

\section{Some questions}

Theorem~\ref{thm:infinitefamilies} answers the questions of whether the
direct construction yields infinitely many fundamental groups and
infinitely many diffeomorphism types.  We would like to address two questions in further work in \cite{AkhmedovTorus}.

\begin{question}
For each of the three cases in Theorem~\ref{thm:classification}, which
finite abelian groups occur as $H_1(Z_\phi;\mathbb Z)$?  For
product-framed gluings Proposition~\ref{prop:productH1} shows that the
order is divisible by
\[
 pq=6,\quad 8,\quad 9
\]
in the $(2,3)$, $(2,4)$, and $(3,3)$ cases, respectively.  Are these
lower bounds sharp?  In particular, can one realize $\mathbb Z/6$ in the
$(2,3)$ case? 

\end{question}

\begin{question}
More generally, for which integers $p,q\geq 2$ and genera $g,h$ does a
mixed sum have the rational cohomology ring of
\[
k(S^2\times S^2)
\]
for some odd $k\geq 3$? Classify all possible quadruples
$(p,q,g,h)$ for each such odd $k \leq 15$?
\end{question}

\begin{question}
Can Luttinger surgeries on Lagrangian tori disjoint from the
multisections alter the boundary maps so that the residual integral
quotients disappear and the final amalgam becomes simply connected?
\end{question}

At the integral level, changing the gluing map alone cannot remove the
residual first homology.  One possibility is to modify the complements
before taking the sum.
The construction in \cite{AkhmedovCohomology} suggests one way this might
happen: the relevant generators there are killed before the final gluing.
Whether suitable Luttinger/torus surgeries in the present complements produce
analogous relations remains open.
\section*{Acknowledgments}
We are grateful to B.~Doug Park for many valuable discussions.  Some of the ideas underlying this work
go back to discussions with him around 2010, and we thank him for his
insights and encouragement.  We also thank Ivan Smith for helpful
correspondence dating back to 2007, when he kindly answered questions
concerning his construction of braided symplectic surfaces in ruled
surfaces. We are also grateful to several colleagues for their encouragement over the years. An LLM-based tool was used to assist with the preparation of portions of the text, including grammar and language editing, and with the generation of the figures. All mathematical content is due to the author, who takes full intellectual responsibility for the content of this paper.

\end{document}